\documentclass[12pt]{article}
\usepackage[paper=a4paper, top=2cm,bottom=2.5cm, left=2cm,right=2cm]{geometry}

\usepackage{graphicx,placeins}
\usepackage{enumerate,enumitem}
\usepackage{amsmath,amssymb,amsthm}
\usepackage{hyperref}

\newtheorem{thm}{Theorem}[section]
\newtheorem{lem}{Lemma}[section]
\newtheorem{prop}{Proposition}[section]

\newcommand{\be}{\begin{equation}}
\newcommand{\ee}{\end{equation}}
\newcommand{\bee}{\begin{equation*}}
\newcommand{\eee}{\end{equation*}}
\newcommand{\bea}{\begin{eqnarray}}
\newcommand{\eea}{\end{eqnarray}}
\newcommand{\bess}{\begin{eqnarray*}}
\newcommand{\eess}{\end{eqnarray*}}

\title{\bf Positive Solutions of Competition Model \\with Saturation\footnote{This work was supported by NSFC Grant 11771110}}
\author{\bf 
	Aung Zaw \uppercase {Myint}$^1$$^2$\footnote{Corresponding Author. E-mail: aungzawmyint@hit.edu.cn}
	Li \uppercase {Li}$^1$ 
	Mingxin \uppercase {Wang}$^1$
}
\date{}

\begin{document}

\maketitle

\vspace{-.5cm}
\noindent
$^1$School of Mathematics, Harbin Institute of Technology, Harbin 150006, PRC\\
$^2$Department of Mathematics, University of Mandalay, Mandalay 05032, Myanmar

\vspace{1.5cm}

{\small
\noindent
{\bf Abstract} \quad 
In this paper, the positive solutions of a diffusive competition model with saturation are mainly discussed. Under certain conditions, the stability and multiplicities of coexistence states are analyzed. And by using the topological degree theory in cones, it is proved that the problem has at least two positive solutions under certain conditions. Finally, investigating the bifurcation of coexistence states emanating from the semi-trivial solutions, some instability and multiplicity results of coexistence state are expressed.


\noindent
{\bf 2010 MR Subject Classification}{\quad 35J57, 35B09, 35B32, 35B35, 92B05}      

}

\section{Introduction}
{\setlength\arraycolsep{2pt}

In this paper, we consider positive solutions of the following diffusive competition
system with homogeneous Dirichlet boundary conditions
\begin{equation}\label{eq:model}
	\begin{cases}
		-\triangle u=u (a-u-b v f(u,v)),\ & x\in\Omega,\\
		-\triangle v=v(c-v-d u f(u,v)), & x\in\Omega,\\
		u=v=0, & x\in\partial\Omega,
	\end{cases}
\end{equation}
where $\Omega$ is a bounded open domain in $\mathbb{R}^N$, $N\geqslant 1$, the boundary $\partial\Omega$ is smooth.

The unknown functions $u$ and $v$ represent densities of two competitive species, respectively. Hence, we only concern about positive solutions of \eqref{eq:model}. The parameters $a, b, c, d$ are positive constants. The functional response
\begin{equation}\label{eq:response}
	f(u,v)=\frac{1}{(1+\alpha u) (1+\beta v)},
\end{equation}
which is proposed by Bazykin \cite{B} to describe the saturation in a predator-prey model. In this functional, $\alpha$ and $\beta$ are non-negative constants.

Regarding to the Bazykin functional response, the existence, multiplicity and uniqueness of positive solutions of the predator-prey model were studied in \cite{WQ, WWG}, where $f(u,v)$ is given by \eqref{eq:response} with positive $\alpha$ and $\beta$.  The predator-prey models with Bazykin functional response and other kinds of functional responses have been studied by many authors, please refer to \cite{LLY, LPW, LWD, NW, PangW, PW, PWC, RA, Z1} for example.

On the other hand, the research on competition model \eqref{eq:model} with Bazykin functional response are few. In \cite{Du3, DZ}, the competition model \eqref{eq:model} was studied for $\alpha=\beta=0$.
In this paper, we obtain existence, stability and multiplicities the positive solutions of problem \eqref{eq:model}. Bifurcations and multiplicities near the semi-trivial solutions are also obtained in the present paper.

The organization of the paper is as follows. In Section \ref{sec:pre}, we present some basic results, including the {\it a priori} estimates and some notations to apply the fixed point index theory. We calculate the fixed point index by using a well-known abstract result (Proposition \ref{prop:index}) in Section \ref{sec:index}. In Section \ref{sec:existence}, we use the topological degree and index theory to study the existence of positive solutions. Using the upper and lower solutions method and eigenvalue theory, we obtain the stability of positive solutions in Section \ref{sec:stab:mult}. Then combining with the topological degree theory in cones, we prove that problem \eqref{eq:model} has at least two positive solutions under suitable conditions. In Section \ref{sec:bif}, applying the topological degree theory and the bifurcation theory established by P.H. Rabinowitz, we investigate the bifurcation of positive solutions emanating from the semi-trivial solutions $(0,\theta_{c}, a_0)$ and $(\theta_{a}, 0, c_0)$, some instability and multiplicity results are also obtained.

\section{Preliminaries}\label{sec:pre}

In this section, we present some known results regarding to the {\it a priori} estimates, some concepts and propositions to apply the fixed point index theory.

Let $\lambda_1(q)$ be the principle eigenvalue of
\[
	\begin{cases}
		-\triangle \phi+q(x)\phi=\lambda\phi, \qquad  & x\in\Omega,\\
		\phi=0, & x\in\partial\Omega,
	\end{cases}
\]
and denote $\lambda_1:=\lambda_1(0)$.

It is well-known that when $a>\lambda_1$, the following problem
\[
	\begin{cases}
		-\triangle u=u (a-u), \qquad & x\in\Omega,\\
		u=0, & x\in\partial\Omega
	\end{cases}
\]
has a unique positive solution $\theta_{a}$. Similarly when $c>\lambda_1$, let $\theta_{c}$ be the unique positive solution of
\[
	\begin{cases}
		-\triangle v=v (c-v), \qquad  & x\in\Omega,\\
		v=0, & x\in\partial\Omega.
	\end{cases}
\]
It is obvious that all possible trivial and semi-trivial solutions of \eqref{eq:model} include $(0,0)$, $(\theta_{a}, 0)$ and $(0,\theta_{c})$. Moreover,
\[
	0 < \theta_{a} \leqslant a, \quad 0 < \theta_{c} \leqslant c, \qquad x\in \Omega,
\]
\[
	\frac{\partial\theta_{a}}{\partial\nu},\frac{\partial\theta_{c}}{\partial\nu}<0, \quad x\in \partial\Omega,
\]
where $\frac{\partial}{\partial \nu}$ is the outer normal derivative.
It can be proved that all non-negative solutions $(u,v)$ of \eqref{eq:model}, $u\not\equiv 0$, $v\not\equiv 0$ must satisfy $u>0$, $v>0$ in $\Omega$. Such positive solutions of \eqref{eq:model} are called {\it coexistence states}.

\begin{thm}[\textit{A priori} estimate]\label{thm:upper:bound}
	Any non-negative solution $(u, v)$ of \eqref{eq:model} satisfies
	\[
		u\leqslant a,\quad v\leqslant c.
	\]
\end{thm}

\begin{proof}
	Since $u$ and $v$ satisfy
	\[
	\begin{cases}
		-\triangle u=u (a-u-b v f(u,v))\leqslant u(a-u),\\
		-\triangle v=v (c-v-d u f(u,v))\leqslant  v(c-v), 		
	\end{cases}	
	\]
	then $u\leqslant \theta_{a} \leqslant a$ and $v\leqslant \theta_{c} \leqslant c$ by applying the maximum principle.
\end{proof}

Now let us introduce some concepts to apply the fixed point index theory, which is essential to get the existence and multiplicity of positive solutions of \eqref{eq:model}.

Let $E$ be a Banach space. A non-empty set $W\subset E$ is called a {\it total wedge} if $W$ is a closed convex set, $\beta W\subseteq W$ for all $\beta\geqslant 0$ and  $E =\overline{W-W}$. For any $y\in W$, we define
\begin{equation}\label{eq:def:WySy}
\begin{aligned}
	W_y &= \{x\in E:\,\exists\ r=r(x)>0,\text{ s.t. } y+rx\in W\},   \\
	S_y&= \{ x\in\overline{W}_y:\, -x\in\overline{W}_y\}.
\end{aligned}
\end{equation}
Then $\overline{W}_y$ is a {\it wedge} containing $W, y, -y$, while $S_y$ is a closed subset of $E$ containing $y$.
A linear compact operator $T: W_y\to W_y$ is said to have \emph{property $\alpha$} on $W_y$ if there exist $t\in (0,1)$ and $w\in W_y\setminus S_y$, such that $(I-tT)w\in S_y$.

Denote $B_\delta^+ (y)=B_\delta(y)\cap W$ for any $y\in W$ and $\delta>0$. Suppose that $F: B_\delta^+ (y)\to W$ is a compact operator and $y$ is an isolated fixed point of $F$. Further assume that $F$ is Fr\'{e}chet differentiable at $y$, so the derivative $F'(y): W_y\to W_y$. We use ${\rm index}_W(F, y)$ to denote the fixed point index of $F$ at $y$ relative to $W$.

\begin{prop}{\rm(\cite{D1, LG, W})}\label{prop:index}
	Suppose that $I-F'(y)$ is invertible on $\overline{W_y}$.
	
	{\rm(i)} If $F'(y)$ has \textit{property $\alpha$}, then ${\rm index}_W(F,y)=0$.
	
	{\rm(ii)} If $F'(y)$ does not have \textit{property $\alpha$}, then
	\[
		{\rm index}_W (F, y)= (-1)^\sigma,
	\]
	where $\sigma$ is the sum of multiplicities of all eigenvalues of $F'(y)$ which are greater than one.
\end{prop}

\begin{prop}{\rm(\cite{D2, LG, Y})}\label{prop:spectrum}
	Let $q\in C(\overline{\Omega}),$ and $M$ be a sufficiently large number such that $M > q(x)$ for all $x\in\overline{\Omega}$.  Denote $L = (M-\triangle)^{-1} \big( M-q(x) \big)$
	
	and $r(L)$ be the spectral radius of $L$.  Then we have:
	
	{\rm(i)} $\lambda_1(q) > 0\Leftrightarrow r(L) < 1$.
	
	{\rm(ii)} $\lambda_1(q) < 0\Leftrightarrow r(L) >1$.
	
	{\rm (iii)} $\lambda_1(q)=0\Leftrightarrow r(L)=1$.
\end{prop}

\section{Calculation of the fixed point index }\label{sec:index}

In this section, we compute the fixed point indices of trivial and semi-trivial solutions of \eqref{eq:model}. The results will be applied to study the existence and multiplicity of coexistence states of problem \eqref{eq:model} in next section.

We introduce the following notations.

$E=X\times X$, where $X=\{ u\in C^1(\bar{\Omega}):\, u|_{\partial\Omega} =0\}$,

$W=K\times K$, where $K=\{ u\in X:\, u\geqslant 0\text{ in }\Omega\}$,

$D=\{ (u,v)\in W:\, u\leqslant 1+a, v\leqslant 1+c\},\quad \mathring{D}=\mathrm{int}  D$.

Obviously $W$ is a total wedge of $E$. By the \textit{a priori} estimates (Theorem \ref{thm:upper:bound}), we know that all possible non-negative solutions of \eqref{eq:model} must lie in $D$. So there exists a sufficiently large constant $M>\lambda_1>0$ such that
\[
	a-u-b v f(u,v)+M >0,\qquad  c-v-d u f(u,v)+M > 0
\]
for any $(u,v)\in\bar{D}$.
Define an operator $F: E\to E$ by
\[
	F(u,v)=(M-\triangle)^{-1}
	\begin{pmatrix}
		u [ a-u-b v f(u,v) ]+Mu\\
		v [ c-v-d u f(u,v) ]+Mv
	\end{pmatrix}.
\]
It is easy to see that $F$ is compact and it maps $D$ to $W$. Also, finding solutions of \eqref{eq:model} is equivalent to finding fixed points of $F$.

For any $t\in [0,1]$, define an operator $F_t: E\to E$ by
\[
	F_t(u,v)=(M-\triangle)^{-1}
	\begin{pmatrix}
		t u [ a-u-b v f(u,v) ]+Mu\\
		t v [ c-v-d u f(u,v) ]+Mv
	\end{pmatrix}.
\]
Obviously, $F_t: [0,1]\times D\to W$ is positive and compact, $F_1=F$.

\begin{lem}\label{lem:convexity}
	Let $u,v\in C^1(\overline{\Omega})$ with $u |_{\partial\Omega}=0$, $v|_{\partial\Omega}\geqslant 0$, $u>0$ in $\Omega$ and $\frac{\partial u}{\partial\nu}\big|\,_{\partial\Omega}<0$. Then there is a positive constant $\varepsilon>0$ such that $u+\epsilon v>0$ in $\Omega$.
\end{lem}
\begin{proof}
	As $u, v\in C^1$ and $\frac{\partial u}{\partial\nu}|_{\partial\Omega} < 0$, there exists $\varepsilon_1 > 0$ such that $\frac{\partial (u+\varepsilon_1 v)}{\partial\nu}|_{\partial\Omega} < 0$. Note that $(u+\varepsilon_1 v)|_{\partial\Omega}\geqslant 0$.
	Thus, there exists a subset $\Omega_0\subset\subset\Omega$ such that $u+\varepsilon_1 v > 0$ in $\Omega\setminus\Omega_0$.  As $u > 0$ on $\overline{\Omega}_0$, there exists $\varepsilon_2 > 0$ such that $ u+\varepsilon_2 v > 0$ on $\overline{\Omega}_0$.  Take $\varepsilon=\mathrm{min}\{\varepsilon_1,\varepsilon_2\}$, and we reach the desired conclusion.
\end{proof}

By the definitions in \eqref{eq:def:WySy} and Lemma \ref{lem:convexity}, it can be proved that

\begin{enumerate}[label={\arabic*)}, leftmargin=32pt]\itemsep-10pt
	\item $\overline{W}_{(0,0)}=K\times K$, $S_{(0,0)}=\{(0,0)\}$,  \\
	\item $\overline{W}_{(\theta_{a},0)}=X\times K$, $S_{(\theta_{a},0)}=X\times\{0\}$,   \\
	\item $\overline{W}_{(0,\theta_{c})}=K\times X$, $S_{(0,\theta_{c})}=\{0\}\times X$.
\end{enumerate}

Now, we are ready to analyze the indices of trivial and semi-trivial solutions: $(0,0)$, $(\theta_{a},0)$  and $(0,\theta_{c})$.

\begin{lem}\label{lem:index1}
	It always holds that $\deg_W (I-F, D)=1$. Suppose $a>\lambda_1$.
	
	{\rm(i)} If $c\not=\lambda_1$, then ${\rm index}_W(F,(0,0))=0$.
	
	{\rm(ii)} If $c>\lambda_1\left(\frac{d\theta_{a}}{1+\alpha\theta_{a}}\right)$, then ${\rm index}_W(F,(\theta_{a},0))=0$.
	
	{\rm(iii)} If $c<\lambda_1\left(\frac{d\theta_{a}}{1+\alpha\theta_{a}}\right)$, then ${\rm index}_W(F,(\theta_{a},0))=1$.
\end{lem}

\begin{proof}
	{\it Step 1\;}  To prove that $\deg_W (I-F, D)=1$.
	
	By Theorem \ref{thm:upper:bound}, $F$ has no fixed point on $\partial D$, so $\deg_W(I-F,D)$ is well-defined. For any $t$, the fixed point of $F_t$ satisfies
	\begin{equation}\label{eq:Ft}
	\begin{cases}
		-\triangle u=t u(a-u-b v f(u,v)),\ & x\in\Omega,\\
		-\triangle v=t v (c-v-d u f(u,v)), & x\in\Omega,\\
		u=v=0, & x\in\partial\Omega.
	\end{cases}
	\end{equation}
	It is obvious that any solution $(u,v)$ of \eqref{eq:Ft} must lie in $\mathring{D}$. The homotopy invariance of degree implies that $\deg_W (I-F_t, D)$ is independent of $t$, so
	\[
		\deg_W (I-F, D)=\deg_W (I-F_1, D)=\deg_W (I-F_0, D).
	\]
	Obviously \eqref{eq:Ft} has only trivial solution $(0,0)$ when $t=0$. Thus,
	\[
		\deg_W (I-F_0, D)={\rm index}_W(F_0, (0,0)).
	\]
	Let
	\[
		L := F_0'(0,0)=(M-\triangle)^{-1}
		\begin{pmatrix}
			M & 0\\ 0 & M
		\end{pmatrix}.
	\]
	It can be proved that $r(L)<1$ by Proposition \ref{prop:spectrum}. Therefore $I-L$ is invertible on $\overline{W}_{(0,0)}$, and $L$ does not have \textit{property $\alpha$} on $\overline{W}_{(0,0)}$. By Proposition \ref{prop:index}, ${\rm index}_W(F_0, (0,0))=1$, so $\deg_W (I-F, D)=1$.
		
	{\it Step 2\;}  To prove (i), ${\rm index}_W(F,(0,0))=0$ if $c\not=\lambda_1$.
	
	Obviously, $F$ is a compact operator on $\overline{D}$ and $(0,0)$ is a fixed point. The Fr\'{e}chet derivative of $F$ at $(0,0)$ is given by
	\[
		F'(0,0)=(M-\triangle)^{-1}
		\begin{pmatrix}
			a+M & 0\\ 0 & c+M
		\end{pmatrix}.
	\]
	Suppose $(\xi,\eta)\in\overline{W}_{(0,0)}$ is a fixed point of $F'(0,0)$, i.e.
	\[
	\begin{cases}
		-\triangle\xi=a\xi, \qquad & x\in \Omega, \\
		-\triangle\eta=c\eta, & x\in \Omega, \\
		\xi = \eta =0, & x\in \partial\Omega.
	\end{cases}
	\]
	Since $a>\lambda_1$, $c\not=\lambda_1$, we have $\xi=\eta\equiv 0$ in $\Omega$. Hence $I-F'(0,0)$ is invertible on $\overline{W}_{(0,0)}$.
		
	Now we claim that $F'(0,0)$ has \textit{property $\alpha$}. Since $a>\lambda_1$, we know by Proposition \ref{prop:spectrum} that
	\[
		r_1 := r((M-\triangle)^{-1}(a+M)) >1.
	\]
	Moreover, $r_1$ is the principle eigenvalue of $(M-\triangle)^{-1}(a+M)$ and the corresponding eigenfunction $\phi>0$. Take $t_0=r_1^{-1}\in (0,1)$, thus
	\[
		(I-t_0 F'(0,0)) (\phi,0)=(0,0)\in S_{(0,0)}.
	\]
	Hence $F'(0,0)$ has \textit{property $\alpha$}. By Proposition \ref{prop:index}, ${\rm index}_W(F,(0,0))=0$.
		
	{\it Step 3\;} To prove (ii).
	
	We already know that $\overline{W}_{(\theta_{a},0)}=X\times K$, $S_{(\theta_{a},0)}=X\times\{0\}$. Thus,
	\[
		\overline{W}_{(\theta_{a},0)}\setminus S_{(\theta_{a},0)}=X\times\{K\setminus\{0\}\}.
	\]
	It can be derived that
	\[
		F'(\theta_{a},0)= (M-\triangle)^{-1}
		\begin{pmatrix}
			a-2\theta_{a} +M & \displaystyle \frac{-b\theta_{a}}{1+\alpha\theta_{a}}\\
			0 &  \displaystyle c -\frac{d\theta_{a}}{1+\alpha\theta_{a}} +M
		\end{pmatrix}.
	\]
		
	Let us first prove that $I-F'(\theta_{a},0)$ is invertible. Assume that $(\xi,\eta)\in\overline{W}_{(\theta_{a},0)}$ satisfies
	\begin{equation}\label{eq:DFthe_a}
	\begin{cases}
		-\triangle\xi = a\xi-2\theta_{a}\xi +\displaystyle\frac{-b\theta_{a}}{1+\alpha\theta_{a}}\eta, \quad  & x\in\Omega, \\[3mm]
		-\triangle \eta = c\eta-\displaystyle\frac{d\theta_{a}}{1+\alpha\theta_{a}}\eta  , & x\in\Omega,  \\[2mm]
		\xi=\eta=0, & x\in\partial\Omega.
	\end{cases}
	\end{equation}
	Easy to see that $\eta\equiv 0$ since $c>\lambda_1( \frac{d\theta_{a}}{1+\alpha\theta_{a}})$. If $\xi$ is a solution to the first line of \eqref{eq:DFthe_a}, then $\xi$ is an eigenfunction of
	\[
		-\triangle\xi+(2\theta_{a}-a)\xi=\lambda\xi,
	\]
	corresponding to eigenvalue $\lambda=0$. Thus
	\begin{equation}\label{eq:temp1}
		0\geqslant\lambda_1(2\theta_{a}-a) >\lambda_1(\theta_{a}-a).
	\end{equation}
	The last inequality is due to the monotonicity of $\lambda_1(q)$ with respect to positive oscillations of $q$.
		
	On the other hand, $\theta_{a}>0$ satisfies $-\triangle\theta_{a}+(\theta_{a}-a)\theta_{a} =0$. So
	\begin{equation}\label{eq:temp2}
		\lambda_1(\theta_{a}-a)=0.
	\end{equation}
	Now we have a contradiction from \eqref{eq:temp1} and \eqref{eq:temp2}. Therefore, $\xi=\eta\equiv 0$. The operator $I-F'(\theta_{a},0)$ is invertible.
	
	Now we prove that $F'(\theta_{a},0)$ has \textit{property $\alpha$}. Since $c>\lambda_1(\frac{d\theta_{a}}{1+\alpha\theta_{a}})$, it can be proved that
	\[
		r_2 := r\left((M-\triangle)^{-1}  (c-\frac{d\theta_{a}}{1+\alpha\theta_{a}}+M)\right)>1.
	\]
	Assume that the principle eigenfunction of $(M-\triangle)^{-1}  (c-\frac{d\theta_{a}}{1+\alpha\theta_{a}}+M)$ is $\phi>0$ in $\Omega$. Take $t_0=r_2^{-1}$, thus
	\[
		(I-t_0F'(\theta_{a},0))
		\begin{pmatrix}
			0\\ \phi
		\end{pmatrix}
		=
		\begin{pmatrix}
			\displaystyle (M-\triangle)^{-1}\frac{t_0b\theta_{a}\phi}{1+\alpha\theta_{a}}\\[2mm] 0
		\end{pmatrix}
		\in S_{(\theta_{a},0)}.
	\]
	The operator $F'(\theta_{a},0)$ has \textit{property $\alpha$}, so ${\rm index}_W(F,(\theta_{a},0))=0$ by Proposition \ref{prop:index}.
		
	{\it Step 4\;} To prove (iii).
	
	Same analysis as in Step 3 implies that $I-F'(\theta_{a},0)$ is invertible on $\overline{W}_{(\theta_{a},0)}$.
	Now we claim that $F'(\theta_{a},0)$ does not have \textit{property $\alpha$}.
	
	From the assumption $c<\lambda_1(\frac{d\theta_{a}}{1+\alpha\theta_{a}})$ and Proposition \ref{prop:spectrum}, we know that
	\begin{equation}\label{eq:rad_c}
		r\left((M-\triangle)^{-1}  (c-\frac{d\theta_{a}}{1+\alpha\theta_{a}}+M)\right)<1.
	\end{equation}
	If $F'(\theta_{a},0)$ has \textit{property $\alpha$}, then there exists $0<t<1$ and $(0,\phi_2)\in\overline{W}_{(\theta_{a},0)}\setminus S_{(\theta_{a},0)}$ such that
	\[
		I-t F'(\theta_{a},0) (0,\phi_2)\in S_{(\theta_{a},0)}.
	\]
	In particular,
	\[
		\left(I-t (M-\triangle)^{-1}  (c-\frac{d\theta_{a}}{1+\alpha\theta_{a}}+M)\right)\phi_2=0.
	\]
	Thus $\phi_2$ is an eigenfunction of $(M-\triangle)^{-1}  (c-\frac{d\theta_{a}}{1+\alpha\theta_{a}}+M)$, and $t^{-1}>1$ is the corresponding eigenvalue, which contradicts \eqref{eq:rad_c}. Hence $F'(\theta_{a},0)$ does not have \textit{property $\alpha$}. By Proposition \ref{prop:index},
	\[
		{\rm index}_W(F,(\theta_{a},0))=(-1)^\gamma,
	\]
	where $\gamma$ is the sum of algebraic multiplicities of all eigenvalues of $F'(\theta_{a},0)$ which are greater than one.
		
	Suppose that $\mu^{-1}>1$ is an eigenvalue of $F'(\theta_{a},0)$ with corresponding eigenfunctions $(\xi,\eta)$. Hence,
	\[
		(M-\triangle)^{-1}
		\begin{pmatrix}
			\displaystyle (a-2\theta_{a} +M)\xi+\frac{-b\theta_{a}}{1+\alpha\theta_{a}}\eta\\[2mm]
			\displaystyle (c -\frac{d\theta_{a}}{1+\alpha\theta_{a}} +M)\eta
		\end{pmatrix}
		=\frac{1}{\mu}
		\begin{pmatrix}
			\xi\\\eta
		\end{pmatrix},
	\]
	i.e.
	\begin{equation}\label{eq:index:temp}
	\begin{cases}
		\displaystyle -\triangle\xi+M\xi= \mu\left((a-2\theta_{a} +M)\xi+\frac{-b\theta_{a}}{1+\alpha\theta_{a}}\eta\right),\  & x\in\Omega,\\[2mm]
		\displaystyle -\triangle\eta+M\eta= \mu (c -\frac{d\theta_{a}}{1+\alpha\theta_{a}} +M)\eta,  & x\in\Omega, \\[1mm]
		\xi=\eta=0, & x\in\partial\Omega.
	\end{cases}
	\end{equation}
	Recall that $M$ was chosen sufficiently large such that $\frac{d\theta_{a}}{1+\alpha\theta_{a}} - c - M<0$.  From the second line of \eqref{eq:index:temp}, if $\eta\not\equiv0$, then
	\[
	\begin{aligned}
		0& \geqslant \lambda_1\left(M-\mu (c -\frac{d\theta_{a}}{1+\alpha\theta_{a}} +M)\right)
		= M + \lambda_1\left(\mu (\frac{d\theta_{a}}{1+\alpha\theta_{a}} - c - M)\right) \\
		& > M + \lambda_1\left( \frac{d\theta_{a}}{1+\alpha\theta_{a}} - c - M \right)
		= -c+\lambda_1\left(\frac{d\theta_{a}}{1+\alpha\theta_{a}}\right).
	\end{aligned}
	\]
	This is a contradiction to the condition that $c<\lambda_1(\frac{d\theta_{a}}{1+\alpha\theta_{a}})$. Thus $\eta\equiv 0$.  Substitute $\eta\equiv 0$ into the first line of \eqref{eq:index:temp}, we have
	\[
		-\triangle\xi+ M\xi= \mu(a-2\theta_{a}+ M)\xi.
	\]
	Thus,
	\[
		0\geqslant\lambda_1\left(M-\mu (a-2\theta_{a}+M)\right)>\lambda_1( 2\theta_{a}-a) >\lambda_1(\theta_{a}-a) =0.	
	\]
	So we have a contradiction and there is no eigenvalue of $F'(\theta_{a},0)$ which is greater than $1$. Hence ${\rm index}_W(F,(\theta_{a},0))=1$. The proof of lemma is complete.
\end{proof}

Similar to Lemma \ref{lem:index1}, we have
\begin{lem}\label{lem:index2}
	Suppose $c>\lambda_1$.
	
	{\rm(i)} If $a\not=\lambda_1$, then ${\rm index}_W(F,(0,0))=0$.
	
	{\rm(ii)} If $ a>\lambda_1(\frac{b\theta_{c}}{1+\beta\theta_{c}})$, then ${\rm index}_W(F, (0,\theta_{c}))=0$.
	
	{\rm(iii)} If $ a<\lambda_1(\frac{b\theta_{c}}{1+\beta\theta_{c}})$, then ${\rm index}_W(F, (0,\theta_{c}))=1$.

\end{lem}

\section{Existence of coexistence states}\label{sec:existence}

In this section, we investigate the existence of positive solutions of \eqref{eq:model} by using the results obtained in Section \ref{sec:index}.

\begin{thm}\label{thm:coexist}
	{\rm(i)} If $a\leqslant\lambda_1$, then the only possible non-negative solutions of \eqref{eq:model} are $(0,0)$ and $(0,\theta_{c})$.

	{\rm(ii)} If $c\leqslant\lambda_1$, then the only possible non-negative solutions of \eqref{eq:model} are $(0,0)$ and $(\theta_{a},0)$.
\end{thm}
\begin{proof}
	If \eqref{eq:model} has a non-negative solution $(u,v)$ with $u\not\equiv 0$, then
	\[
	\begin{cases}
		-\triangle u=u (a-u-bvf(u,v)),\  & x\in\Omega,\\
		u=0, & x\in\partial\Omega.
	\end{cases}
	\]
	Thus
	\[
		0=\lambda_1(u+bvf(u,v)-a) > -a+\lambda_1.
	\]
	Therefore, $u,v\geqslant 0$, $u\not\equiv 0$ implies that $a>\lambda_1$, (i) holds. Part (ii) can be proved similarly.
\end{proof}

\begin{thm}\label{thm:exist}
	{\rm(i)} If \eqref{eq:model} admits coexistence states, then $a>\lambda_1$, $c>\lambda_1$.

	{\rm(ii)} If $  a>\lambda_1(\frac{b\theta_{c}}{1+\beta\theta_{c}})$, $  c>\lambda_1(\frac{d\theta_{a}}{1+\alpha\theta_{a}})$, then \eqref{eq:model} admits at least one coexistence state.

	{\rm(iii)} If $ \lambda_1<a<\lambda_1(\frac{b\theta_{c}}{1+\beta\theta_{c}})$, $ \lambda_1<c<\lambda_1(\frac{d\theta_{a}}{1+\alpha\theta_{a}})$, then \eqref{eq:model} admits at least one coexistence state.
\end{thm}
\begin{proof}
	(i) holds due to Theorem \ref{thm:coexist}. It remains to prove (ii) and (iii).
	
	If $  a>\lambda_1(\frac{b\theta_{c}}{1+\beta\theta_{c}})$, $  c>\lambda_1(\frac{d\theta_{a}}{1+\alpha\theta_{a}})$, it follows from Lemma \ref{lem:index1} and  \ref{lem:index2}  that
	\begin{align*}
		& \deg_W (I-F, D)-{\rm index}_W(F,(0,0))\\
		& \quad -{\rm index}_W(F,(\theta_{a},0))- {\rm index}_W(F,(0,\theta_{c})) = 1 - 0 - 0 - 0 = 1.
	\end{align*}
	Thus \eqref{eq:model} admits at least one coexistence state.
		
	Similarly when $ \lambda_1<a<\lambda_1(\frac{b\theta_{c}}{1+\beta\theta_{c}})$, $ \lambda_1<c<\lambda_1(\frac{d\theta_{a}}{1+\alpha\theta_{a}})$, we know that
	\begin{align*}
		& \deg_W (I-F, D)-{\rm index}_W(F,(0,0))\\
		& \quad -{\rm index}_W(F,(\theta_{a},0))- {\rm index}_W(F,(0,\theta_{c})) = 1 - 0 - 1 - 1 = -1
	\end{align*}
	by Lemma \ref{lem:index1} and \ref{lem:index2}. Hence \eqref{eq:model} also admits at least one coexistence state.
\end{proof}

\section{Stability and multiplicity of coexistence states}\label{sec:stab:mult}

We will study the stability and multiplicity of positive solutions of \eqref{eq:model} for $\alpha$ or $\beta$ suitably large. We first present an asymptotic result.

\begin{lem}\label{lem:large:albe:existence}
	Let $a, c >\lambda_1$. For any small $\varepsilon$, $0<\varepsilon<\min(a-\lambda_1, c-\lambda_1)$,  there is a constant $\overline{\alpha}(\varepsilon)$ (or $\overline{\beta}(\varepsilon)$) such that as $\alpha\geqslant\overline{\alpha}(\varepsilon)$ (or $\beta\geqslant\overline{\beta}(\varepsilon)$), problem \eqref{eq:model} has at least one positive solution $(u, v)$ satisfying
	\begin{equation}\label{eq:upp:low}
		\theta_{a-\varepsilon}\leqslant u\leqslant\theta_{a},\qquad
		\theta_{c-\varepsilon}\leqslant  v\leqslant\theta_{c}.
	\end{equation}
\end{lem}
\begin{proof}
	From the similarities of two competitive species, the effects of $\alpha$ and $\beta$ in $f(u,v)$ are equivalent, so we only need to prove the lemma for $\alpha$.
	
	Denote $\underline{U}= (\underline{u},\underline{v})=(\theta_{a-\varepsilon},\theta_{c-\varepsilon})$ and $\overline{U}=(\overline{u},\overline{v})=(\theta_{a},\theta_{c})$. It is obvious that
	\[
		u(a -u- b v f(u,v)),\qquad  v(c-v -d u f(u,v))
	\]
	are Lipschitz continuous w.r.t. $u$ and $v$ for $\underline{U}\leqslant (u,v)\leqslant\overline{U}$.  If we can prove that $\overline{U}$ and $\underline{U}$ are the upper and lower solutions of \eqref{eq:model}, then \eqref{eq:model} has at least one coexistence state $U= (u,v)$ that satisfies \eqref{eq:upp:low}.
	
	To show that $(\overline{U}, \underline{U})$ is a pair of upper and lower solutions, it suffices to prove the following inequalities:
	\begin{eqnarray}
		\triangle\overline{u}+\overline{u}(a-\overline{u}-b\underline{v} f(\overline{u},\underline{v}))&\leqslant & 0,  \label{eq:ul:sol:1}\\
		\triangle\overline{v}+\overline{v}(c-\overline{v}-d\underline{u}f(\underline{u},\overline{v}))&\leqslant & 0,\label{eq:ul:sol:2}\\
		\triangle\underline{u}+\underline{u}(a -\underline{u}- b\overline{v} f(\underline{u},\overline{v}))&\geqslant & 0, \label{eq:ul:sol:3}\\
		\triangle\underline{v}+\underline{v}(c- \underline{v}-d\overline{u}f(\overline{u},\underline{v}))&\geqslant & 0.   \label{eq:ul:sol:4}
	\end{eqnarray}
	
	The validity of \eqref{eq:ul:sol:1} and \eqref{eq:ul:sol:2} is obvious, now we prove inequalities \eqref{eq:ul:sol:3} and \eqref{eq:ul:sol:4} for $\alpha$ sufficiently large. A direct computation gives
	\[
	\begin{aligned}
		\triangle\underline{u}+\underline{u}(a-\underline{u}-b\overline{v}f(\underline{u},\overline{v})=\theta_{a-\varepsilon}(\varepsilon- b\theta_{c}f(\theta_{a-\varepsilon},\theta_{c})),\\
		\triangle\underline{v}+\underline{v}(c-\underline{v}-d\overline{u}f(\overline{u},\underline{v}))=\theta_{c-\varepsilon}(\varepsilon-d\theta_{a}f(\theta_{a},\theta_{c-\varepsilon})).
	\end{aligned}
	\]
	As $\theta_{a}=\theta_{c}=0$ on  $\partial\Omega$, it is clear that the inequalities hold near $\partial\Omega$.  Noting that, as $\alpha\rightarrow\infty$,
	\[
		b\theta_{c} f(\theta_{a-\varepsilon},\theta_{c})\rightarrow 0,\quad   d\theta_{a}f(\theta_{a},\theta_{c-\varepsilon})\rightarrow 0, \quad x\in \Omega' \subset\subset \Omega
	\]
	uniformly on any compact subset $\Omega'$ of $\Omega$, so inequalities \eqref{eq:ul:sol:3} and \eqref{eq:ul:sol:4} also hold in $\Omega'$ provided that $\alpha$ is sufficiently large. The theorem is proved.
\end{proof}

\begin{thm}\label{thm:large:albe:li:sta}
	Assume that $a,c>\lambda_1$. Then \eqref{eq:model} has at least one positive solution which is linearly stable and non-degenerate if $\alpha$ (or $\beta$) is suitably large.
\end{thm}
\begin{proof}
	From the similarities of two competitive species, it suffices to prove the theorem for $\alpha$.
	
	Take a positive sequence $\{\varepsilon_i\}$, $\varepsilon_i\to 0^+$. By Lemma \ref{lem:large:albe:existence}, there exist $\overline{\alpha}(\varepsilon_i)$ suitably large such that when $\alpha_i\geqslant\overline{\alpha}(\varepsilon_i)$, the problem \eqref{eq:model} has at least one positive solution, denoted by $(u_{\alpha_i}, v_{\alpha_i})$, satisfying
	\begin{equation}\label{ineq:uv}
		\theta_{a-\varepsilon_i}\leqslant u_{\alpha_i}\leqslant\theta_{a},\quad
		\theta_{c-\varepsilon_i}\leqslant v_{\alpha_i}\leqslant\theta_{c}.
	\end{equation}
	We claim that such positive solutions $(u_{\alpha_i},v_{\alpha_i})$ are also linearly stable if $i$ is suitably large.
	
	Assume the contrary that there exists $\alpha_i\to\infty$, $\mu_i$ satisfying $\mathrm{Re}\mu_i\leqslant 0$, and $(\xi_i,\eta_i)\not\equiv 0$ satisfying $\|\xi_i\|_{L^2}^2+\|\eta_i\|_{L^2}^2=1$, such that
	\begin{equation}\label{eq:stat}
	\begin{cases}
		-\triangle\xi_i-(a-2u_i- f_i)\xi_i+f_i^*\eta_i=\mu_i\xi_i,  \quad &  x\in\Omega,\\
		-\triangle\eta_i+g_i\xi_i-(c-2v_i- g_i^*)\eta_i=\mu_i\eta_i,  & x\in\Omega,\\
		\xi_i=\eta_i=0, & x\in\partial\Omega,
	\end{cases}
	\end{equation}
	where $(u_i, v_i)=(u_{\alpha_i }, v_{\alpha_i})$, and
	\[
	\begin{aligned}
		f_i=\frac{bv_i}{(1+\alpha_i u_i)^2(1+\beta v_i)},\qquad    f_i^*=\frac{bu_i}{(1+\alpha_i u_i)(1+\beta v_i)^2}, \\
		g_i=\frac{dv_i}{(1+\alpha_i u_i)^2(1+\beta v_i)},\qquad    g_i^*=\frac{du_i}{(1+\alpha_i u_i)(1+\beta v_i)^2}.
	\end{aligned}
	\]
	Multiply the first line of (\ref{eq:stat}) by $\overline{\xi}_i$, the second line by $\overline{\eta}_i$, add them together, then integrate over $\Omega$, we obtain
	\[
	\begin{aligned}
		\mu_i=&\int_{\Omega} (|\nabla\xi_i|^2+|\nabla\eta_i|^2) d x +\int_{\Omega}\left(|\xi_i|^2(f_i+2u_i-a)+f_i^*\eta_i\overline{\xi}_i\right) d x   \\
		&+\int_{\Omega}\left(g_i\xi_i\overline{\eta}_i+ |\eta_i|^2(g_i^*+2v_i- c)\right) d x.
	\end{aligned}
	\]	
	Note that $\mathrm{Re}\mu_i\leqslant 0$, $\|\xi_i\|_{2}^2+\|\eta_i\|_{2}^2=1$, and $u_i$, $v_i$ are bounded. It follows from the above equality that $\mathrm{Re}\mu_i$ and $\mathrm{Im}\mu_i$ are both bounded. Thus $\{\mu_i\}_{i=1}^\infty$ is bounded.  We may assume that $\mu_i\rightarrow\mu$ and $\mathrm{Re}\mu\leqslant 0$.  By the standard regularity theory and bootstrap argument for elliptic equations, it can be derived that $\xi_i$ and $\eta_i$ are bounded in $W^{2, p} (\Omega)$ for any $p>n$. Thus, there are subsequences of $\{\xi_i\}$ and $\{\eta_i\}$, denoted by themselves, such that $\xi_i\rightarrow\xi,\eta_i\rightarrow\eta$ in $W^{1, p}(\Omega)$.
	
	Since $u_i$, $v_i$ are bounded by \eqref{ineq:uv}, let $\varepsilon_i\rightarrow 0$ in \eqref{eq:stat}, it follows that $(\mu,\xi,\eta)$ satisfies
	\begin{equation}\label{eq:12}
	\begin{cases}
		-\triangle\xi-\xi(a-2\theta_{a})=\mu\xi,  & x\in\Omega,\\
		-\triangle\eta-\eta(c-2\theta_{c})=\mu\eta,  & x\in\Omega,\\
		\xi=\eta=0,  & x\in\partial\Omega
	\end{cases}
	\end{equation}
	in weak sense. Notice that $\xi,\eta\in W^{1,p}(\Omega)\hookrightarrow C^\alpha (\overline{\Omega})$, we observe that \eqref{eq:12} holds in classical sense according to the regularity theory of elliptic equations. Therefore, $\mu$ is real and $\mu\leqslant 0$.
	
	If $\xi\not\equiv 0$, then $\mu$ is an eigenvalue of the problem
	\[
	\begin{cases}
		-\triangle\phi+(2\theta_{a}-a)\phi=\mu\phi,\ & x\in\Omega,\\
		\phi=0, & x\in\partial\Omega.
	\end{cases}
	\]
	Hence $0\geqslant\mu\geqslant\lambda_1(2\theta_{a}-a)>\lambda_1(\theta_{a}-a)=0$, which is a contradiction. Thus $\xi\equiv 0$.  Similarly, $\eta\equiv 0$.  This contradicts to the assumption that $\|\xi_i\|_{2}^2+\|\eta_i\|_{2}^2=1$. The proof of theorem is complete.
\end{proof}

\begin{thm}\label{thm:large:albe:multi}
	Suppose that $a>\lambda_1, c>\lambda_1$.

	{\rm(i)} If $a <\lambda_1(\frac{b\theta_{c}}{1+\beta\theta_{c}})$, then there exists a large positive constant $\overline{\alpha}$ such that \eqref{eq:model} has at least two positive solutions as $\alpha\geqslant\overline{\alpha}$.
	
	{\rm(ii)} If $c <\lambda_1(\frac{d\theta_{a}}{1+\alpha\theta_{a}})$, then there exists a large positive constant $\overline{\beta}$ such that \eqref{eq:model} has at least two positive solutions as $\beta\geqslant\overline{\beta}$.
\end{thm}
\begin{proof}
	By the similarities of two competitive species, we only need to prove (i). By Theorem \ref{thm:large:albe:li:sta}, \eqref{eq:model} has at least one positive solution $(\tilde{u},\tilde{v})$ which is linearly stable and non-degenerate if $\alpha$ is sufficiently large. This implies that the operator $I- F'(\tilde{u},\tilde{v})$ is invertible on $\overline{W}_{(\tilde{u},\tilde{v})}$ and $F'(\tilde{u},\tilde{v})$ has no real eigenvalue which is greater than one. Note that $W_{(\tilde{u},\tilde{v})}=E=S_{(\tilde{u},\tilde{v})}$. It can be proved that $F'_{(\tilde{u},\tilde{v})}$ does not have \textit{property $\alpha$} and $\mathrm{index}_W(F, (\tilde{u},\tilde{v}))=1$ by Proposition \ref{prop:index}.

	Suppose $ \eqref{eq:model}$ admits only one coexistence state $(\tilde{u},\tilde{v})$, we can apply Lemma \ref{lem:index1}, Lemma \ref{lem:index2} (iii) and the additivity of the fixed point indices, to deduce that
	\begin{eqnarray*}
		1&=&\deg_W (I-F,D)  \\
		&= &{\rm index}_W(F,(0,0))+{\rm index}_W(F,(\theta_{a},0))\\
		&&+\;{\rm index}_W(F,(0,\theta_{c}))+{\rm index}_W(F,(\tilde{u},\tilde{v}))  \\
		&= &0+0+1+1.
	\end{eqnarray*}
	This is a contradiction and the theorem is proved.
\end{proof}

\section{Bifurcation, instability and multiplicity of coexistence states}\label{sec:bif}

In this section, we discuss the bifurcation of positive solutions by using respectively $ a $ and $ c $ as the main bifurcation parameters, and study the multiplicity and stability of coexistence states. Given an operator $F$, we use $\mathcal{N}(F)$ (or $\mathcal{N}F$) and $\mathcal{R}(F)$ (or $\mathcal{R}F$) to denote the kernel and range of $F$, respectively. We first introduce some notations which will be used in Theorem \ref{thm:bif} to describe the bifurcations.

Firstly, we regard $ a $ as a bifurcation parameter and suppose that all other parameters are fixed. If $ c >\lambda_1$, it is obvious that the problem \eqref{eq:model} has semi-trivial non-negative solutions: $\{(0,\theta_{c}, a): a\in\mathbb{R}\}$. By linearizing \eqref{eq:model} at $ (0,\theta_{c})$, we obtain the following eigenvalue problem:
\begin{equation}\label{eq:bifur}
\begin{cases}
	-\triangle\xi + \displaystyle \frac{b\theta_{c}}{1+\beta\theta_{c}}\xi-a\xi= \mu\xi,  & x\in\Omega,\\[4mm]
	-\triangle\eta +\displaystyle\frac{d\theta_{c}}{1+\beta\theta_{c}}\xi+2\theta_{c}\eta-c\eta=\mu\eta,\  & x\in\Omega, \\[1mm]
	\xi=\eta=0, & x\in\partial\Omega.
\end{cases}
\end{equation}
If $\mu=0$ is the principle eigenvalue of \eqref{eq:bifur}, which occurs at
\[
	a=a_0 :=\lambda_1(\frac{b\theta_{c}}{1+\beta\theta_{c}}),
\]
we will prove that $(0,\theta_{c},a_0)$ is a bifurcation point in Theorem \ref{thm:bif}.

Let $\Phi_a$ be the unique positive solution of
\begin{equation}\label{eq:Phia}
\begin{cases}
	-\triangle\Phi_a +\displaystyle\frac{b\theta_{c}}{1+\beta\theta_{c}}\Phi_a=  a_0\Phi_a, \ & x\in\Omega,\\
	\Phi_a=0, & x\in\partial\Omega, \\[2mm]
	\int_\Omega\Phi_a^2=1.
\end{cases}
\end{equation}
Since
\[
	\lambda_1(2\theta_{c}-c) >\lambda_1(\theta_{c}-c)=0,
\]
$-\triangle+2\theta_{c} -c$ is invertible, and $(-\triangle+2\theta_{c} -c)^{-1}$ maps positive functions to positive functions because of the maximum principle.  Define
\begin{equation}\label{eq:Psia}
	\Psi_a := d (-\triangle+2\theta_{c} -c)^{-1}\left(\frac{-\theta_{c}}{1+ \beta\theta_{c}}\Phi_a\right),
\end{equation}
then $\Phi_a>0$, $\Psi_a<0$ in $\Omega$ and $(\Phi_a, \Psi_a)$ satisfies \eqref{eq:bifur} with $\mu=0$, $a=a_0$.

Similar to the above argument, we can regard $c$ as the bifurcation parameter and suppose that all other parameters are fixed. Let
\[
	c_0:=\lambda_1(\frac{d\theta_{a}}{1+\alpha\theta_{a}})
\]
and $\Psi_c$ be the unique positive solution of
\begin{equation}\label{eq:Psic}
\begin{cases}
	\displaystyle -\triangle\Psi_c+\frac{d\theta_{a}}{1+\alpha\theta_{a}}\Psi_c=  c_0\Psi_c,\  & x\in\Omega,\\
	\Psi_c=0, & x\in\partial\Omega, \\[2mm]
	\int_\Omega\Psi_c^2=1.
\end{cases}
\end{equation}
Define
\begin{equation}\label{eq:Phic}
	\Phi_c := b (-\triangle+2\theta_{a} -a)^{-1}\left(\frac{-\theta_{a}}{1+ \alpha\theta_{a}}\Psi_c\right).
\end{equation}
It can be proved that $(-\triangle+2\theta_{a} -a)^{-1}$ is invertible, $\Phi_c<0$, $\Psi_c>0$ in $\Omega$ and $(\Phi_c,\Psi_c)$ satisfies \eqref{eq:bifur} with $\mu=0$, $c=c_0$.

With the constants $a_0, c_0$ and functions $\Phi_a,\Psi_a,\Phi_c,\Psi_c$ defined above, we have the following results regarding the bifurcation of positive solutions of \eqref{eq:model} from $(0,\theta_{c}, a_0)$ and $(\theta_{a}, 0, c_0)$, respectively.

\begin{thm}\label{thm:bif}
	{\rm(i)}\, Assume $ c >\lambda_1$ and $ a=a_0$. Then $(0,\theta_{c},a_0)$ is a bifurcation point of positive solution of \eqref{eq:model}.  Moreover, for $0<s\ll 1$, the bifurcating positive solution $(u(s), v(s), a(s))$ of \eqref{eq:model} emanating from $(0,\theta_{c},a_0)$ takes the form
	\begin{equation}\label{eq:bif:co:sol}
	\begin{cases}
		u(s)= s\Phi_a+O(s^2),\\
		v(s)=\theta_{c}+s\Psi_a+O(s^2),\\
		a(s)= a_0+a_1 s+O(s^2),
	\end{cases}
	\end{equation}
	where $\Phi_a,\Psi_a$ are defined by \eqref{eq:Phia}, \eqref{eq:Psia}, and
	\begin{equation}\label{eq:a_1}
		a_1=\int_\Omega \Phi_{a}^3 d x +b\int_\Omega\frac{1}{(1+\beta\theta_{c})^2}\Phi_{a}^2\Psi_a d x -b\alpha\int_\Omega\frac{\theta_{c}}{1+\beta\theta_{c}}\Phi_{a}^3 d x .
	\end{equation}
	
	{\rm(ii)}\, Assume $ a >\lambda_1 $ and $ c=c_0$. Then $(\theta_{a}, 0, c_0)$ is a bifurcation point of positive solution of \eqref{eq:model}.  Moreover, for $0<s\ll 1$, the bifurcating positive solution $(u(s), v(s), c(s))$ of \eqref{eq:model} emanating from $(\theta_{a}, 0, c_0)$ takes the form
	\[
	\begin{cases}
		u(s)=\theta_{a}+s\Phi_c+O(s^2),\\
		v(s)= s\Psi_c+O(s^2),\\
		c(s)= c_0+c_1 s+O(s^2),
	\end{cases}
	\]
	where $\Phi_c,\Psi_c$ are defined by \eqref{eq:Psic}, \eqref{eq:Phic}, and
	\begin{equation}\label{eq:c_1}
		c_1=\int_\Omega \Psi_{c}^3 d x +d\int_\Omega\frac{1}{(1+\alpha\theta_{a})^2}\Psi_{c}^2\Phi_c d x -d\beta\int_\Omega\frac{\theta_{a}}{1+\alpha\theta_{a}}\Psi_{c}^3 d x .
	\end{equation}
\end{thm}
\begin{proof}
	Due to the similarities of two competitive species, we only need to prove (i). Define an operator $F : E\times\mathbb{R}\rightarrow E$ by
	\[
		F(u,v,a) =
		\begin{pmatrix}
			\triangle u+u (a-u-b v f(u,v))\\
			\triangle v+v (c-v-d u f(u,v))
		\end{pmatrix}.
	\]
	It is obvious that $F_a(0,\theta_{c}, a_0)=0$. For any $(\xi,\eta)\in E$, a direct calculation yields
	\[
		F_{(u,v)}(u,v,a)
		\begin{pmatrix}
		\xi\\
		\eta
		\end{pmatrix}
		=
		\begin{pmatrix}
			\displaystyle \triangle\xi+ (a-2u-\frac{b v f(u,v)}{1+\alpha u})\xi-\frac{bu f(u,v)}{1+\beta v}\eta\\
			\displaystyle \triangle\eta-\frac{dv f(u,v)}{1+\alpha u}\xi+(c-2v-\frac{d u f(u,v)}{1+\beta v})\eta 	
		\end{pmatrix}.
	\]
	Hence,
	\[
		F_{(u,v)}(0,\theta_{c},a_0)
		\begin{pmatrix}
			\xi\\
			\eta
		\end{pmatrix}
		=
		\begin{pmatrix}
			\displaystyle \triangle\xi+ (a_0 -\frac{b\theta_{c}}{1+\beta\theta_{c}})\xi\\
			\displaystyle \triangle\eta-\frac{d\theta_{c}}{1+\beta\theta_{c}}\xi+(c- 2\theta_{c})\eta 	
		\end{pmatrix}.
	\]
	
	{\it Step 1\;} We shall prove that
	\begin{equation}\label{eq:kernel}
		\mathrm{dim}(\mathcal{N}F_{(u,v)}(0,\theta_{c},a_0))=1,\qquad\mathcal{N}F_{(u,v)}(0,\theta_{c}, a_0) =\mathrm{span}\{(\Phi_a,\Psi_a)\}.
	\end{equation}
	In fact, if there exists $(\xi,\eta)\in E$ such that $F_{(u,v)}(0,\theta_{c}, a_0)(\xi,\eta)=(0, 0)$, then
	\[\begin{cases}
		\displaystyle \triangle\xi +(a_0-\frac{b\theta_{c}}{1+\beta\theta_{c}})\xi=0,  \quad  & x\in\Omega,\\
		\displaystyle \triangle\eta-\frac{d\theta_{c}}{1+\beta\theta_{c}}\xi+(c-2\theta_{c})\eta=0,   & x\in\Omega,\\
		\xi=\eta=0,  & x\in\partial\Omega.
	\end{cases}
	\]
	It follows from the first line and $a_0 =\lambda_1(\frac{b\theta_{c}}{1+\beta\theta_{c}})$ that $\xi\in\mathrm{span}\{\Phi_a\}$, i.e. $\xi= k\Phi_a $ for some constant $k\in \mathbb{R}$.
	Since the operator $\triangle+c-2\theta_{c}$ is invertible, we have
	\[
		\eta=(-\triangle+2\theta_{c}-c)^{-1}\left(\frac{-d\theta_{c}}{1+\beta\theta_{c}}\xi\right) = k\Psi_a.
	\]
	Therefore, $(\xi,\eta)\in\mathrm{span}\{(\Phi_a,\Psi_a)\}$.
	
	{\it Step 2\;} To prove that $\mathrm{Codim} (\mathcal{R}F_{(u,v)}(0,\theta_{c}, a_0))= 1$.
	
	In fact, if $(\xi^*,\eta^*)\in\mathcal{R}F_{(u,v)} (0,\theta_{c}, a_0)$, then there exists $(\xi,\eta)\in E$ such that
	\begin{equation}\label{eq:bf}
	\begin{cases}
		\displaystyle \triangle\xi+ (a_0 -\frac{b\theta_{c}}{1+\beta\theta_{c}})\xi=\xi^*,\quad & x\in\Omega,\\
		\displaystyle \triangle\eta-\frac{d\theta_{c}}{1+\beta\theta_{c}}\xi+(c-2\theta_{c})\eta=\eta^*, &  x\in\Omega,\\
		\xi=\eta=0, & x\in\partial\Omega.
	\end{cases}
	\end{equation}
	As $\Phi_a$ being the unique positive solution of \eqref{eq:Phia}, we have
	\[
		\int_{\Omega}\Phi_a\xi^* d x  =0,
	\]
	and thus $(\xi^*,\eta^*)$ is orthogonal to $(\Phi_a, 0)$.
	
	Conversely, if $(\xi^*,\eta^*)$ is orthogonal to $(\Phi_a, 0)$, then the first equation of \eqref{eq:bf}
	has a solution $\xi$. Therefore, the second equation of \eqref{eq:bf} admits a
	solution $\eta$ since $\triangle+ c-2\theta_{c}$ is invertible. Thus $(\xi^*,\eta^*)\in\mathcal{R}F_{(u,v)}(0,\theta_{c},a_0)$ and $\mathrm{Codim} (\mathcal{R}F_{(u,v)}(0,\theta_{c}, a_0))= 1$.
	
	{\it Step 3\;} Since $\mathcal{R}F_{(u,v)}(0,\theta_{c}, a_0)$ is orthogonal to $(\Phi_a, 0)$, we have
	\[
		F_{(u,v),a} (0,\theta_{c},a_0)(\Phi_a,	\Psi_a)=(\Phi_a, 0) \not\in\mathcal{R}F_{(u,v)}(0,\theta_{c}, a_0).
	\]
	Applying the bifurcation theorem in \cite{CR}, we arrive at the desired conclusions of (i). Actually, by \eqref{eq:kernel}, $(u(s),v(s),a(s))$ can be expressed in the form \eqref{eq:bif:co:sol}. From the first line of \eqref{eq:model}, we have
	\[
		-\int\Phi_a\triangle u d x =\int\Phi_a u (a-u-bvf(u,v)) d x .
	\]
	On the other hand, integrating the left hand side by parts, we have
	\[
		-\int\Phi_a\triangle u d x  =-\int u\triangle\Phi_a d x =\int u (a_0-\frac{b\theta_{c}}{1+\beta\theta_{c}})\Phi_a d x .
	\]
	Combining the above two equations, we have
	\begin{equation}\label{eq:a_1:temp}
		\int\Phi_a u (a-u-bvf(u,v)) d x =\int u (a_0-\frac{b\theta_{c}}{1+\beta\theta_{c}})\Phi_a d x .
	\end{equation}
	Equation \eqref{eq:a_1} can then be achieved by substituting \eqref{eq:bif:co:sol} into \eqref{eq:a_1:temp}.
	Part (ii) can be proved similarly.
\end{proof}

By Theorem \ref{thm:bif}, we know that $(0,\theta_{c}, a_0)$ and $(\theta_{a}, 0, c_0)$ are bifurcation points of coexistence states for any $b, d > 0$. The following Theorem \ref{thm:nondeg} and Theorem \ref{thm:nondeg2} discuss the stability and multiplicity of the coexistence states, which  bifurcates from $(0,\theta_{c}, a_0)$ (or $(\theta_{a},0, c_0)$), when $0<s,d\ll 1$ (or $0<s,b\ll 1$).

\begin{thm}\label{thm:nondeg}
	Let $c >\lambda_1$ and $\int_\Omega\Phi_{a}^3\left(1-\frac{\alpha b\theta_{c}}{1+\beta\theta_{c}}\right) d x \not=0$. If $0<s,d\ll 1$, then the coexistence state $(u(s), v(s), a(s))$ bifurcating from $(0,\theta_{c},a_0)$ is non-degenerate.  If $\int_\Omega\Phi_{a}^3\left(1-\frac{\alpha b\theta_{c}}{1+\beta\theta_{c}}\right) > 0$, then $(u(s), v(s))$ is linearly stable; If $\int_\Omega\Phi_{a}^3\left(1-\frac{\alpha b\theta_{c}}{1+\beta\theta_{c}}\right)< 0$, then it is linearly unstable.
	
	Moreover, if $a_1$ in \eqref{eq:a_1} is negative, then \eqref{eq:model} has at least two coexistence states for $0<d\ll 1$, $a<a_0$ and near $a_0$.
\end{thm}
\begin{proof}
	For convenience, simply denote $a(s)=a$ and $(u(s),v(s))=(u,v)$. Then, the corresponding linearized problem at $(u,v)$ can be written as
	\[
		\mathcal{L}(s,d)(\xi,\eta)=\mu (s,d) (\xi,\eta),
	\]
	where
	\[
		\mathcal{L}(s,d)=
		\begin{pmatrix}
		\displaystyle -\triangle-a+2u+\frac{bv f(u,v)}{1+\alpha u} &\displaystyle \frac{bu f(u,v)}{1+\beta v}\\
		\displaystyle \frac{dv f(u,v)}{1+\alpha u} &\displaystyle -\triangle-c+2v+\frac{du f(u,v)}{1+\beta v}
		\end{pmatrix}.
	\]
	As $s, d\rightarrow 0^+$, it is easy to see that
	\[
		\lim_{s\rightarrow 0^+}\mathcal{L}(s,d) =
		\begin{pmatrix}
			\displaystyle -\triangle-a_0+\frac{b\theta_{c}}{1+\beta\theta_{c}}   &  0\\
			0 &  -\triangle-c+2\theta_{c}
		\end{pmatrix}=:\mathcal{L}_0.
	\]
	Because $ a_0=\lambda_1(\frac{b\theta_{c}}{1+\beta\theta_{c}})$ and $\lambda_1(2\theta_{c}-c) >\lambda_1(\theta_{c}-c)=0$, we know that $0$ is the first eigenvalue of $\mathcal{L}_0$ with the corresponding eigenfunction $(\Phi_a, 0)$.  Moreover, all the real parts of the other eigenvalues of $\mathcal{L}_0$ are positive and apart from 0.  According to the perturbation theory of linear operators \cite{Km, Kt}, it can be proved that, when $0 < s,d\ll 1$, $\mathcal{L}(s,d)$ has a unique eigenvalue $\mu(s,d)$ satisfying $\lim_{s,d\rightarrow 0^+}\mu (s,d)=0$ and all other eigenvalues of $\mathcal{L}(s,d)$ have positive real parts and apart from 0.  In the following, we shall simply denote $\mathcal{L}(s,d)=\mathcal{L}$ and $\mu (s,d)=\mu$.
	
	To determine the sign of $\mathrm{Re}\mu$ for $0 < s,d\ll 1$, we set $(\xi,\eta)$ be the corresponding eigenfunction to $\mu$ such that $(\xi,\eta)\rightarrow (\Phi_a, 0)$ as $s,d\rightarrow 0^+$. Then $(\xi,\eta)$ satisfy
	\begin{equation}\label{eq:bf:sta}
	\begin{cases}
		\displaystyle -\triangle\xi- \left( a-2u-\frac{b v f(u,v)}{1+\alpha u} \right)\xi+\frac{b u f(u,v)}{1+\beta v}\eta=\mu\xi,\quad  & x\in\Omega,\\
		\displaystyle -\triangle\eta+\frac{dv f(u,v)}{1+\alpha u}\xi- \left( c-2v-\frac{d u f(u,v)}{1+\beta v} \right)\eta=\mu\eta,   & x\in\Omega,\\
		\xi=\eta=0,   & x\in\partial\Omega.
	\end{cases}
	\end{equation}
	Multiplying the first equation of \eqref{eq:bf:sta} by $ u $ and integrating over $\Omega $, we obtain
	\[
		-\int_\Omega u\triangle\xi d x -\int_\Omega (a-2u-\frac{b v f(u,v)}{1+\alpha u}) u\xi d x +\int_\Omega\frac{b u^2 f(u,v)}{1+\beta v}\eta d x =\mu\int_\Omega u\xi d x .
	\]
	On the other hand, multiplying the first equation of \eqref{eq:model} by $\xi$, and integrating over $\Omega $, we obtain
	\[
		-\int_\Omega u\triangle\xi d x  =-\int_\Omega\xi\triangle u d x =\int_\Omega\xi u (a-u-b v f(u,v)) d x .	
	\]
	By combining the above two equations, it yields
	\begin{equation}\label{ineq:sta}
		\mu\int_\Omega u\xi d x =\int_\Omega\xi u^2 (1-\frac{b v\alpha f(u,v)}{1+\alpha u}) d x +\int_\Omega\frac{b u^2 f(u,v)}{1+\beta v}\eta d x .
	\end{equation}
	Noting that $(u, v)=(s\Phi_a+O(s^2),\theta_{c}+s\Psi_a+O(s^2))$ and $(\xi,\eta)\rightarrow (\Phi_a, 0)$ as $s,d\rightarrow 0^+$.  Divide \eqref{ineq:sta} by $s^2$ and let $s, d\rightarrow 0^+$, it is deduced that
	\begin{equation}\label{eq:lim}
		\lim_{s,d\rightarrow 0^+}\frac{\mu}{s}
		=\int_\Omega\Phi_{a}^3\left(1-\frac{\alpha b\theta_{c}}{1+\beta\theta_{c}}\right) d x
		\neq 0.
	\end{equation}
	Therefore $\mathrm{Re}\mu\neq 0$ for $s,d$ sufficiently small.  Because all other eigenvalues of $\mathcal{L}$ have positive real parts and apart from 0, the coexistence state bifurcating from $(0,\theta_{c}, a_0)$ is non-degenerate.
	
	We have proved in the above that, when $ 0< s,d\ll 1 $, the eigenvalues of $\mathcal{L}$ have positive real parts and are apart from $0$ except for $\mu (s,d)$.  Thus, the linear stability of the bifurcation coexistence state $(u(s), v(s))$ is determined completely by the sign of the real part of $\mu (s,d)$.  From the limit \eqref{eq:lim}, we see that the real part of $\mu (s,d)$ and $\int_\Omega\Phi_{a}^3(1 -\frac{\alpha b\theta_{c}}{1+\beta\theta_{c}}) d x  $ have the same sign for $0 < s,d\ll 1$. This completes the first assertion of the theorem.
	
	When $a_1<0$, we assume the contrary that \eqref{eq:model} has only one coexistence state $(\hat{u},\hat{v})$ for $0<s,d\ll 1$. From the first part of the proof, $(\hat{u},\hat{v})$ must be the coexistence state bifurcating from $(0,\theta_{c},a_0)$, i.e. $(\hat{u},\hat{v})=(u(s),v(s))$, which is non-degenerate, and $a=a_0+a_1 s+O(s^2)$. Since $a_1<0$, thus $a<a_0$ for $0<s\ll 1$. Therefore, the operator
	\[
		I-F_{(u,v)}(\hat{u},\hat{v}):\overline{W}_{(\hat{u},\hat{v})}\to  \overline{W}_{(\hat{u},\hat{v})}
	\]
	is invertible on $\overline{W}_{(\hat{u},\hat{v})}$. Since $\overline{W}_{(\hat{u},\hat{v})}=E=S_{(\hat{u},\hat{v})}, F_{(u,v)}(\hat{u},\hat{v})$ does not have \textit{property $\alpha$} on $\overline{W}_{(\hat{u},\hat{v})}$. Hence, ${\rm index}_W (F,(\hat{u},\hat{v}))=\pm 1$ by Proposition \ref{prop:index}.
	
	Since $c >\lambda_1$ and $\lambda_1 < a < a_0$ for $s\ll 1$, we have
	 \[
		\begin{aligned}
		1=&\deg_W (I-F,D)\\
		= &{\rm index}_W(F,(0,0))+{\rm index}_W(F,(\theta_{a},0))\\
		&+{\rm index}_W(F,(0,\theta_{c}))+{\rm index}_W(F,(\hat{u},\hat{v}))\\
		=& 0+0+1\pm 1
		\end{aligned}
	\]
	by Lemma \ref{lem:index1} and \ref{lem:index2}. This is a contradiction. Therefore, there exist at least two coexistence sates of \eqref{eq:model}.
\end{proof}

Similarly, we have the following theorem concerning the stability and multiplicity of coexistence states which bifurcates from $(\theta_{a}, 0, c_0)$.

\begin{thm}\label{thm:nondeg2}
	Assume $a >\lambda_1$ and $\int_\Omega\Psi_{c}^3\left(1-\frac{\beta d\theta_{a}}{1+\alpha\theta_{a}}\right) d x \not=0$. If $0<s, b\ll 1$, then the coexistence state $(u(s), v(s), c(s))$ bifurcating from $(\theta_{a}, 0, c_0)$ is non-degenerate. If $\int_\Omega\Psi_{c}^3\left(1-\frac{\beta d\theta_{a}}{1+\alpha\theta_{a}}\right) > 0$, then $(u(s), v(s))$ is linearly stable; If $\int_\Omega\Psi_{c}^3\left(1-\frac{\beta d\theta_{a}}{1+\alpha\theta_{a}}\right) < 0$, then it is linearly unstable.
	
	Moreover, if $c_1$ in \eqref{eq:c_1} is negative, then \eqref{eq:model} has at least two coexistence states for $0<b\ll 1$, $c<c_0$ and near $c_0$.
\end{thm}

\section{Conclusions}

In this paper, we study a diffusive competition model \eqref{eq:model} with saturation, where the functional response is in the form
\[
	f(u,v)=\frac{1}{(1+\alpha u) (1+\beta v)}.
\]
The trivial and semi-trivial solutions include $(0,0)$, $(\theta_{a},0)$ and $(0,\theta_{c})$. Of course, the coexistence states with $u,v>0$ in $\Omega$ have more practical interest.

For arbitrarily fixed parameters $a,b,c,d,\alpha$ and $\beta$, the existence results of positive solutions of \eqref{eq:model} (Theorems \ref{thm:coexist} and \ref{thm:exist}) can be described in Figure \ref{fig:exist}.

For any $a,c>\lambda_1$, we can choose $\alpha$ or $\beta$ suitably large (Theorems \ref{thm:large:albe:li:sta} and \ref{thm:large:albe:multi}) such that \eqref{eq:model} has at least one coexistence state which is linear stable. When $a<\lambda_1(\frac{b\theta_{c}}{1+\beta\theta_{c}})$ and $\alpha\gg 1$, or $c<\lambda_1(\frac{d\theta_{a}}{1+\alpha\theta_{a}})$ and $\beta\gg 1$, there exist at least two coexistence states of \eqref{eq:model}. See Figure \ref{fig:mul1} and \ref{fig:mul2}.

\begin{figure}[htb]
	\centering
	\includegraphics[width=0.65\textwidth]{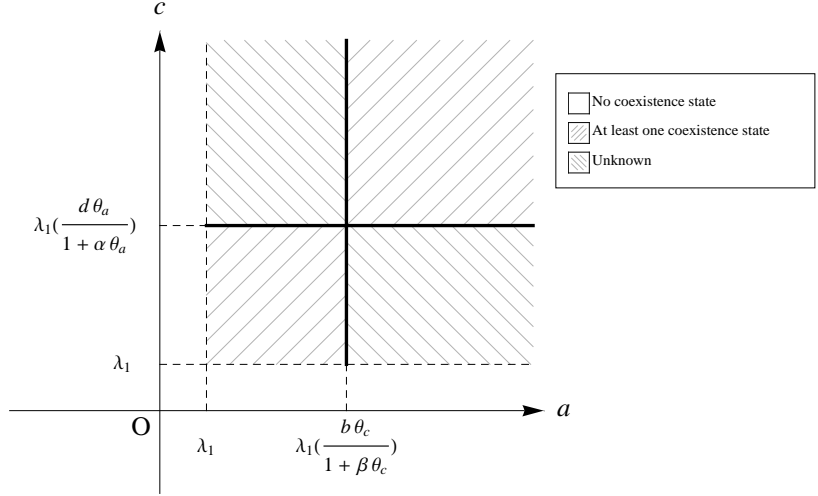}
	\caption{The existence of coexistence states and bifurcation lines}\label{fig:exist}
\end{figure}

\begin{figure}[htb]
	\centering
	\includegraphics[width=0.65\textwidth]{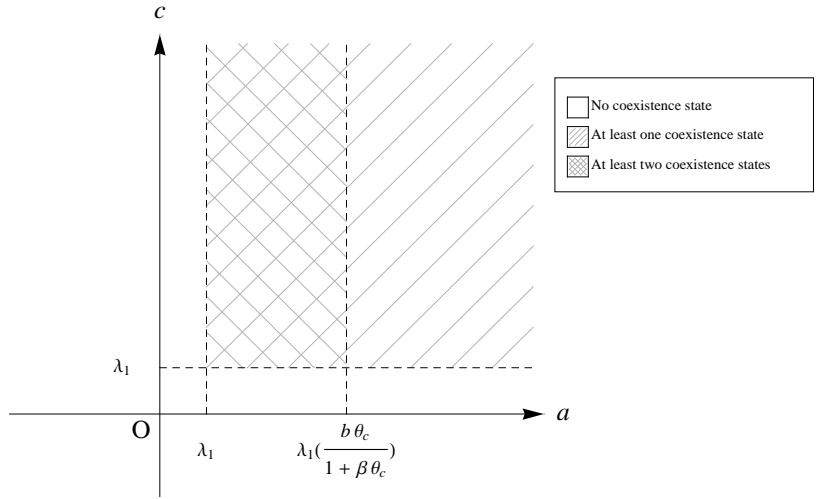}
	\caption{For $\alpha\gg 1$, the existence and multiplicity of coexistence states}\label{fig:mul1}
\end{figure}

\begin{figure}[htb]
	\centering
	\includegraphics[width=0.65\textwidth]{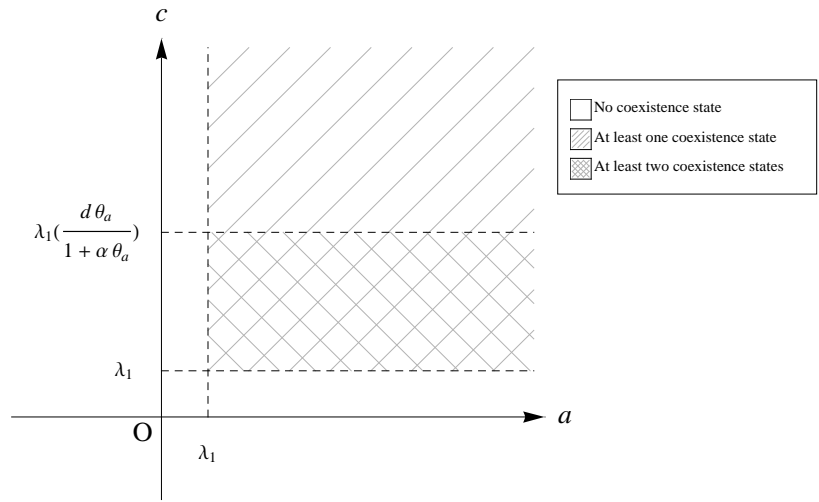}
	\caption{For $\beta\gg 1$, the existence and multiplicity of coexistence states}\label{fig:mul2}
\end{figure}

The bifurcation results are given in Section \ref{sec:bif}. The bifurcation occurs on the solid lines in Figure \ref{fig:exist}, where a coexistence state emanates from semi-trivial solutions $(\theta_{a},0)$ or $(0,\theta_{c})$ (Theorem \ref{thm:bif}). Finally, some stability and multiplicity results are presented in Theorems \ref{thm:nondeg} and \ref{thm:nondeg2}.

\FloatBarrier

\end{document}